\documentclass[leqno,11pt]{amsart}
\usepackage{amssymb, amsmath,latexsym,amsfonts,amsbsy, amsthm,esint}
\usepackage{color}
\usepackage{caption}
\usepackage{hyperref}
\usepackage[foot]{amsaddr}

\let\pa=\partial
\let\al=\alpha
\let\b=\beta
\let\g=\gamma
\let\d=\delta

\let\lam=\lambda
\let\r=\rho
\let\s=\sigma
\let\f=\frac

\let\ka=\kappa
\let\om=\omega
\let\G= \Gamma
\let\D=\Delta

\let\S=\Sigma

\let\Om=\Omega

\let\e=\varepsilon
\let\pa=\partial
\let\va=\varphi
\let\ri=\rightarrow
\let\na=\nabla
\let\th=\theta

\def\non{\nonumber}

\def\di{\mathrm{div}\,}

\newcommand{\beq}{\begin{equation}}
\newcommand{\eeq}{\end{equation}}
\newcommand{\beqo}{\begin{equation*}}
\newcommand{\eeqo}{\end{equation*}}
\newcommand{\ben}{\begin{eqnarray}}
\newcommand{\een}{\end{eqnarray}}
\newcommand{\beno}{\begin{eqnarray*}}
\newcommand{\eeno}{\end{eqnarray*}}

\newtheorem{thm}{Theorem}[section]

\newtheorem{theorem}{Theorem}[section]
\newtheorem{definition}[theorem]{Definition}
\newtheorem{lemma}[theorem]{Lemma}
\newtheorem{proposition}[theorem]{Proposition}
\newtheorem{corol}[theorem]{Corollary}

\newtheorem{Theorem}{Theorem}[section]

\newtheorem{Remark}[Theorem]{Remark}

\theoremstyle{remark}

\newtheorem{rmk}{Remark}[section]

\newcommand{\dist}{\mathrm{dist}}
\newcommand{\BR}{\mathbb{R}}

\newcommand{\ch}{\mathcal{H}}

\newcommand{\cl}{\mathcal{L}}
\newcommand{\ta}{\tilde{A}}

\begin{document}
\title[Asymptotic behavior of free boundary]{Asymptotic behavior of the free interface for entire vector minimizers in phase transitions}

\author{Nicholas D. Alikakos$^1$}
\address{$^1$Department of Mathematics, University of Athens (EKPA), Panepistemiopolis, 15784 Athens, Greece}
\email{nalikako@math.uoa.gr}

\author{Zhiyuan Geng$^2$}
\address{$^2$Basque Center for Applied Mathematics, Alameda de Mazarredo 14
48009 Bilbao, Bizkaia, Spain}
\email{zgeng@bcamath.org}

\author{Arghir Zarnescu$^{2,3,4}$}
\address{$^3$IKERBASQUE, Basque Foundation for Science, Plaza Euskadi 5 48009 Bilbao, Bizkaia, Spain}
\address{$^4$``Simion Stoilow" Institute of the Romanian Academy, 21 Calea Grivi\c{t}ei, 010702 Bucharest, Romania}
\email{azarnescu@bcamath.org}

\thanks{This research is supported by the Basque Government through the BERC 2018-2021 program and by the Spanish State Research Agency through BCAM Severo Ochoa excellence accreditation SEV-2017-0718 and through project PID2020-114189RB-I00 funded by Agencia Estatal de Investigaci\'{o}n (PID2020-114189RB-I00 / AEI / 10.13039/501100011033).}

\begin{abstract}
We study globally bounded entire minimizers $u:\BR^n\ri\BR^m$ of Allen-Cahn systems for potentials $W\geq 0$ with $\{W=0\}=\{a_1,...,a_N\}$ and $W(u)\sim |u-a_i|^\al$ near $u=a_i$, $0<\al<2$. Such solutions are, over large regions, identically equal to some zeroes of the potential $a_i$'s. We establish the estimates
\beqo
\cl^n(I_0\cap B_r(x_0))\leq c_1r^{n-1},\quad \ch^{n-1}(\pa^* I_0\cap B_r(x_0))\geq c_2r^{n-1}, \quad r\geq r_0(x_0)
\eeqo
for the \emph{diffuse interface} $I_0:=\{x\in\BR^n: \min_{1\leq i\leq N}|u(x)-a_i|>0\}$ and the \emph{free boundary} $\pa I_0$. Furthermore, if $\al=1$ we establish the upper bound
\beqo
\ch^{n-1}(\pa^* I_0\cap B_r(x_0))\leq  c_3r^{n-1}, \quad r\geq r_0(x_0).
\eeqo
\end{abstract}

\maketitle

\section{Introduction and description of the main results}

The object of study in the present paper is a class of entire solutions of the system
\beq\label{Euler-Lagrange equation}
\Delta u-W_u(u)=0,\quad u:\BR^n\ri\BR^m,
\eeq
$n,m\in\mathbb{N}^+$, where $W:\BR^m\ri\BR$ is a phase transition potential that is nonnegative and vanishes only on a finite set $\{W=0\}=:A=\{a_1,...a_N\}$ for some distinct points $a_1,...,a_N\in\BR^m$ that represent the phases of a substance which can exist in $N\geq 2$ different equally preferred phases.

The system \eqref{Euler-Lagrange equation} is the Euler-Lagrange equation corresponding to the Allen-Cahn free energy functional
\beq\label{AC functional}
J_D(v)=\int_D \left( \f12|\na v|^2+W(v) \right)\,dx.
\eeq

We restrict ourselves to maps $u\in W_{loc}^{1,2}(\BR^n,\BR^m)\cap L^\infty(\BR^n,\BR^m)$, which minimize $J$ subject to their Dirichlet data,
\beq\label{def of minimizer}
J_D(u+v)\geq J_D(u),\quad \forall v\in W_0^{1,2}(D,\BR^m)\cap L^\infty(D,\BR^m)
\eeq
for any open, bounded Lipschitz set $D\subset \BR^n$. We call such maps \emph{entire minimizers} of $J$, and note that they clearly satisfy \eqref{Euler-Lagrange equation}, under appropriate regularity hypothesis on $W$.

In relation to the phase transition interpretation, an interesting and difficult problem is the existence of multi-phase solutions, that are solutions with the following geometrical properties:

\emph{There is an $\hat{N}\in \mathbb{N}$ with $2\leq \hat{N}\leq N$, $\hat{a}_1,...,\hat{a}_{\hat{N}}\in A$, a small $\g>0$ and open sets $\Om_1,...,\Om_{\hat{N}}$ such that}
\beq\label{partition by various phases}
\BR^n=I\cup\left( \bigcup\limits_{j=1}^{\hat{N}}\Om_j \right)
\eeq
\emph{with $I$ being a set of thickness $O(1)$ and}
\beq\label{in each phase}
|u(x)-\hat{a}_j|\leq \g,\quad\forall x\in \Om_j, \text{ for some small }\g
\eeq
The set $I$ plays the role of a diffuse interface that separates the coexisting phases. Understanding the geometrical structure of this diffuse interface is a major point in the study of such solutions \cite{F}.

In the scalar case $m=1$, for $W\in C^2(\BR^m,\BR)$, $a_i$ nondegenerate (i.e. $\f{\pa^2 W}{\pa u^2}>0$ at the minima), and for $N=2$ which is the natural choice, there is a rich literature and many important results. We list some of these works and organize them in two groups: papers that address various general aspects (see \cite{ft,ks,gurtin,sternberg,cc,cc2,bcn,modica1,modica2,farina}); and papers that are motivated by a celebrated conjecture of De Giorgi (see \cite{degiorgi,mm,gg,ac,s,wei,ct,dkw0}) where a relationship of $I$ with minimal surfaces, and in particular with hyperplanes in low dimensions, is established. The reader could also consult the expository papers \cite{fv,savin,s0,dkw}.

In the vector case $m\geq 2$ and $N\geq 3$, the structure of $I$ is not expected to be planar in any dimension $n\geq 2$ but rather linked to the minimal cones\footnote{The complete classification of the minimal cones in $\BR^n$ is known only for n=3 \cite{taylor}. We refer to Section 7 in the expository paper of David \cite{david} and references therein.} in $\BR^n$. This kind of solutions are called \emph{junctions} and have been shown to exist for $n=m=2$ and $n=m=3$ with $\hat{N}=N=3$ and $\hat{N}=N=4$ respectively, but so far only under symmetry hypotheses. Specifically, the existence is established for:
\begin{itemize}
\item $n=2$ with respect to the symmetries of the equilateral triangle \cite{bgs},
\item $n=3$ with respect to the symmetries of the tetrahedron \cite{gs},
\end{itemize}
and with \eqref{def of minimizer} verified only in their respective equivariance classes. We refer to \cite{book-afs} where the symmetric case is covered in detail for general reflection point groups and also for lattices.

For subquadratic potentials that behave like $|u-a|^\al$ near $a\in A$, $\al\in(0,2)$, the interface is less diffuse as we explain below. More precisely we consider potentials that satisfy

\begin{enumerate}
\itemsep0.6em
   \item[H1.] $W\in C(\BR^m, [0,\infty))\cap C^2_{loc}(\BR^m\backslash A)$ with $\{W=0\}=A=\{a_1,\dots,a_N\}$, $N\geq 2$.
   \item[H2.] Set
   \beq\label{def of r0}
   r_0:=\f12\min\limits_{1\leq i\neq j\leq N}|a_i-a_j|.
   \eeq
 For any $a_i\in A$, $W(u)$ can be written as
       \beq\label{formula for W near ai}
       W(u)=|u-a_i|^\al g_i(u) \;\;\text{ in }u\in B_{r_0}(a_i)
       \eeq
       for some function $g_i(u)\in C^2(B_{r_0}(a_i), [c_i,\infty))$ where $c_i$ is a positive constant.
\end{enumerate}

For such subquadratic potentials, \eqref{in each phase} is replaced by
\beq\label{in each phase free bdy}
u(x)=\hat{a}_j,\quad\forall  x\in \Om_j.
\eeq
Indeed, in this case the entire minimizer possesses a free boundary and the phases $\hat{a}_j$ are attained \cite{agz}, while for $\al=2$, the solution converges exponentially at infinity to the phases. The subquadratic assumption can be thought as a reduction which simplifies without changing the essential features of this type of solution. In a suitable limit $\al\ri0$, \eqref{AC functional} becomes the Alt-Caffarelli functional.

We now define the appropriate analog of the set $I$ in \eqref{partition by various phases} for vector analog of the subquadratic potentials.
\begin{definition}\label{diffuse interface}
Let $0<\g_0<\f12\min\limits_{i\neq j}|a_i-a_j|$ fixed. Set
\beq\label{def of delta}
\d(x)=\dist(u(x),\{W=0\})
\eeq
where $\dist$ stands for the Euclidean distance. Let $0<\g<
\g_0$ and assume $\g_0<\sup\limits_{\BR^n} \d(x)$. We define the set
\beq\label{def of diffuse interface}
I_\g=\{x\in\BR^n:\d(x)\geq \g\}.
\eeq
\end{definition}

For entire minimizers satisfying $\|u\|_{L^\infty(\BR^n,\BR^m)}<\infty$, $\|\na u\|_{L^\infty}(\BR^n,\BR^m)<\infty$, and $0<\al\leq 2$, the following estimate is known (see \cite[Lemma 5.5]{book-afs})
\beq\label{est of I_g}
c_1(\g)r^{n-1}\leq \cl(I_\g\cap B_r(x_0))\leq c_2(\g)r^{n-1},\quad r\geq r(x_0),
\eeq
where $x_0\in\BR^n$ is arbitrarily chosen, $r(x_0)$ is a positive constant depending on $x_0$, and the constants $c_i(\g)>0 \,(i=1,2)$ is independent of $x_0,\,r$.

Clearly $I_{\g_1}\subset I_{\g_2}$ if $\g_1>\g_2$ and we define
\beq\label{I_0}
\lim\limits_{\g\ri 0} I_\g=I_0=\{x\in\BR^n: \d(x)>0\}=: \text{Diffuse Interface}.
\eeq

The constants $c_1(\g),\, c_2(\g)$ in \eqref{est of I_g} degenerate as $\g\ri 0$ (see \cite{book-afs}) and so no useful information can be obtained for $I_0$ out of \eqref{est of I_g}.

Our main result in the present paper concerns certain global facts on $I_0$ that can be summarized in the following theorem.

\begin{theorem}\label{main}
Let $\al\in(0,2)$, $u:\BR^n\ri\BR^m$ be a bounded entire minimizer with $I_{\g_0}\neq \varnothing$ (i.e. $u\not\equiv a_i$), and assume H1, H2 above. Then there exists a radius $r_0>0$ and positive constants $c,\,c_1,\,c_2$, which only depend on $u$ but not on $r$, such that for $r\geq r_0$ the following estimates hold:
\begin{align}
\label{measure of I_0} &\cl^n(I_0\cap B_r(0^n))\leq cr^{n-1},\\
\label{free bdy lower bdd}  &\ch^{n-1}(\pa^* I_0\cap B_r(0^n))\geq c_1r^{n-1},
\end{align}
where $\pa^*$ denotes the De Giorgi reduced boundary.

If furthermore $\al=1$ we have the upper bound estimate
\beq\label{est free bdy}
\ch^{n-1}(\pa^* I_0\cap B_r(0^n))\leq c_2r^{n-1}.
\eeq
\end{theorem}


The analogs of the ``minimal cone" solutions in $\BR^2$ and $\BR^3$ mentioned above as well as the cylindrical triple junction cone in $\BR^3$  have been shown to exist also for $0\leq \al<2$ and possess free boundary (see \cite[Theorem 1 and Proposition 4]{agz})\footnote{For $\al=0$, $J(u)=\int\f12|\na u|^2+\chi_{\{u\in S_A\}}$, where $S_A$ denotes the interior of the convex hull of $A=\{a_1,...,a_N\}$. It is the vectorial analog of the Alt-Caffarelli functional. }. Combining Theorem \ref{main} above with that result we obtain the following
\begin{corol}\label{corol with symmetry}
Under H1 and H2, $0\leq \al <2$, for the equivariant triangle junction in $\BR^2$ and the equivariant quadruple junction in $\BR^3$ the estimates \eqref{measure of I_0}, \eqref{free bdy lower bdd} and \eqref{est free bdy} hold for any $r\geq r_0>0$.
 \end{corol}

\begin{Remark}
We note that the assumption on the existence of minimizers as above is not restrictive as such nonconstant entire minimizers $U:\BR^1\ri\BR^m$ have been shown to exist in \cite{ms} under the hypothesis of continuity on $W$ together with the mild condition
\beq\label{mild condition}
\sqrt{W(u)}\geq f(|u|),\text{ for some nonnegative }f:(0,+\infty)\ri\BR\text{ s.t. }\int_0^{+\infty}f(r)\,dr=+\infty,
\eeq
that allows even decay to zero of the potential at infinity.

A more convenient sufficient condition for the existence of such nonconstant entire minimizers (called \emph{connections}) is (see \cite{fgn})
\beq\label{condition at infty}
\liminf\limits_{|u|\ri\infty} W(u)>0.
\eeq
\end{Remark}

In \cite{cc}, for the scalar two phase problem, the entire range of potentials $F_0=\chi_{\{|u|<1\}}$, $F(u)=(1-u^2)^{\al/2}$ ($0<\al\leq 1$) was already introduced. As $\al\ri 0$, the minimizers get increasingly localized. In particular for $\al=0$ the connections are affine functions. For these reasons we expect the construction of entire minimizers mentioned in Corollary \ref{corol with symmetry} above, under no symmetry hypotheses, to be more accessible for singular potentials.

From the point of view of regularity, our problem can be regarded as a generalization of the more simplified model which studies the minimization problem of the functional
\beq\label{simplified model}
\int_{\Om}\left( \f12|\na u|^2+|u|^\al\right)\,dx,\quad u:\BR^n\ri\BR^m, \; \al\in(0,2).
\eeq
In the scalar case, i.e. $m=1$, one recovers the two phase free boundary problem
\beqo
\D u= \al(u^+)^{\al-1}-\al(u^-)^{\al-1}.
\eeqo
which has been extensively studied under various conditions and settings, see for example \cite{fs} for the case $\al\in(1,2)$ and \cite{lp} for the 2D problem with $\al\in (0,1)$. One can also refer to \cite{lqt,ls} for the optimal regularity theory for a functional with a more general form of the potential $W$. When $\al=1$, the problem becomes the so-called two-phase obstacle-type problem. The regularity of the solution as well as the free boundary regularity has been summarised in detail in \cite{book-psu}.

For the vector-valued case, i.e. $m\geq 2$, the minimization problem \eqref{simplified model} was investigated in \cite{asuw} for $\al=1$ and \cite{fsw} for $\al\in(1,2)$, where the authors studied the regularity of the minimizers and the asymptotic behavior for the minimizer near the ``regular" point of the free boundary. Also we would like to mention the works \cite{csy,mtv} studying the vector-valued Bernoulli free boundary problem, where the potential function $W(u)$ takes the form $Q^2(x)\chi_{\{|u|>0\}}$. Such a problem is quite close to the $\al=0$ case of the functional \eqref{simplified model}. All these works mainly focus on the behavior and regularity of the free boundary.

The biggest difference between our functional \eqref{AC functional} and the simplified one \eqref{simplified model} is that the potential $W(u)$ possesses more than one global minimum. Furthermore, our emphasis is not so much on the local behavior of the free boundary, but on its asymptotic behavior on the large domain $B_R(0^n)$ as $R\ri \infty$.
We note that usually the existence of the free boundary is forced by the Dirichlet boundary condition given on $\pa\Om$. However, in this paper we do not assume any boundary condition and the existence of the free boundary results from the assumption that $u$ is an entire bounded non-constant minimizer. 

To prove the main Theorem \ref{main}, we first show in Section \ref{section 2} that the Euler-Lagrange equation is satisfied by the minimizer $u$ via the regularity of the solution. Then using the regularity property and a non-degeneracy lemma, we prove the first part of the theorem (upper bound estimate for the $\cl^n$ measure of $I_0$) by introducing a method of dividing $B_R$ into identical smaller sub-cubes and classifying all the cubes according to how much measure of the ``contact set" they contain. This is done in Theorem \ref{first part main theorem} from Section \ref{section 3}. We also demonstrate in Theorem \ref{two phase existence} the coexistence of at least two phases for the minimizer at each large scale. In Section \ref{growth est} we derive a Weiss-type formula and then give a growth estimate for the minimizer near the free boundary. These results from Section \ref{growth est} are used in Section \ref{section 5}, where we estimate the $\mathcal{H}^{n-1}$ measure of the free boundary when $\al=1$ and prove the second part of Theorem \ref{main}. Note that currently our method of estimating the $\ch^{n-1}$ measure of the free boundary from above cannot be generalized for arbitrary $\al\in (0,2)$, and we hope to solve this problem in the future.

\section{Regularity of the minimizer}\label{section 2}

We first prove the optimal regularity of the entire minimizer $u$ for $\al\in(0,1)$.

\begin{proposition}\label{reg 0<al<=1}
Suppose $0<\al<1$ and let $u$ be an entire minimizer of the energy functional $J$ satisfying $|u(x)|\leq M$. Assuming H1, H2. Then we have
\beqo
u\in C_{loc}^{1,\beta}(\BR^n,\,\BR^m),
\eeqo
where $\beta$ is a constant defined by $\beta=\f{\al}{2-\al}$. In particular, this $C^{1,\b}$ regularity is sharp.
\end{proposition}

\begin{rmk}
We would like to note that in the following proof, unless being specifically stated, all the constants denoted by $C$ only depend on the upper bound $M$, the potential function $W$ and the dimensions $m,\, n$.
\end{rmk}

\begin{proof}
For any ball $B_R(x)\subset \BR^n$, let $v$ be the harmonic function (which means each $v_i$ is harmonic) in $B_R(x)$ such that $u=v$ on $\pa B_R$. Since $v$ is harmonic, then $\na v$ satisfies a Campanato type growth condition
\beq\label{campanato}
\int_{B_\rho}|\na v-(\na v)_\rho|^2\,dx\leq C(\f{\r}{R})^{n+2}\int_{B_R}|\na v-(\na v)_R|^2\,dx
\eeq

For $\r<R$, we deduce
\beq\label{decay of nabla u}
\begin{split}
\int_{B_\r}|\na u-(\na u)_\r|^2\,dx&\leq C\int_{B_\r}|\na v-(\na v)_\r|^2\,dx+C\int_{B_R}|\na u-\na v|^2\,dx\\
&\leq C(\f{\r}{R})^{n+2}\int_{B_R}|\na v-(\na v)_R|^2\,dx+C\int_{B_R}|\na u-\na v|^2\,dx.
\end{split}
\eeq
For the second term, we use the minimization property to obtain
\beq\label{mini}
\int_{B_R}|\na u-\na v|^2\,dx= \int_{B_R}\left(|\na u|^2-|\na v|^2\right)\,dx
\leq\int_{B_R}2\left(W(v)-W(u)\right)\,dx.
\eeq

By hypothesises H1, H2 and $\|u\|_{L^\infty}\leq M$, we can easily verify that $W$ can be written as
\beq\label{formula for W}
W(u)=\left(\prod\limits_{i=1}^N |u-a_i|^\al \right)g(u),
\eeq
where $g$ is a function such that
\beq\label{condition on g}
g\in C^2(\overline{B}_M), \quad g(u)\geq C \text{ for some constant }C>0.
\eeq

When $\al\in(0,1)$, since both $v$ and $u$ are uniformly bounded, we compute
\beq\label{diff w alpha}
\begin{split}
& W(v)-W(u)\\
=&\left(\prod\limits_{i=1}^N|v-a_i|^\al\right) g(v)-\left(\prod\limits_{i=1}^N|u-a_i|^\al\right) g(u)\\
\leq&\left(\prod\limits_{i=1}^N(|u-a_i|^\al+|u-v|^\al)\right)\left(g(u)+C|u-v|\right)-\left(\prod\limits_{i=1}^N|u-a_i|^\al\right) g(u)\\
\leq & C|u-v|^\al,
\end{split}
\eeq
where $C$ is a constant only depending on $W, M, n, m$.
By H\"{o}lder, Poincar\'{e} and Young, we have
\begin{align}
\non&\int_{B_R}|u-v|^\al\,dx \\
\non\leq & C(\int_{B_R} |u-v|^2\,dx)^{\f{\al}{2}} R^{n(1-\f{\al}{2})}\\
\label{thirdline}\leq & C(\int_{B_R}|\na(u-v)|^2\,dx)^{\f{\al}{2}}R^{n(1-\f{\al}{2})+\al}\\
\non\leq& \delta \int_{B_R}|\na(u-v)|^2\,dx+C(\delta)R^{n+\f{2\al}{2-\al}} 
\end{align}
where $\delta$ is a suitably chosen small number. Combining \eqref{mini}, \eqref{diff w alpha} and \eqref{thirdline}, we obtain
\beq\label{decay of nabla u-v}
\int_{B_R}|\na u-\na v|^2\,dx\leq CR^{n+2\beta},\quad \beta:=\f{\al}{2-\al}
\eeq
Therefore from \eqref{campanato}, \eqref{decay of nabla u} and \eqref{decay of nabla u-v} we get the following Campanato type decay estimate for the minimizer $u$
\beq\label{ucampanato}
\int_{B_\rho}|\na u-(\na u)_\r|^2\,dx\leq c(\f{\r}{R})^{n+2}\int_{B_R}|\na u-(\na u)_R|^2\,dx+CR^{n+2\beta}
\eeq
By a standard iteration argument we conclude that there exists a constant $C$ such that
\beq\label{decayrate}
\int_{B_\rho}|\na u-(\na u)_\r|^2\,dx\leq C\r^{n+2\beta},
\eeq
which by the Morrey-Campanato theory implies $u\in C^{1,\beta}$.


To show that this $C^{1,\beta}$ estimate is sharp, we need some results that will be proved later. Firstly by Theorem \ref{first part main theorem} and its proof, we know there exists $a_i\in A$ such that $\cl^n(\{u(x)=a_i\})>0$ and $\{u(x)=a_i\}$ contains interior points. Let $x_1\in\{u(x)=a_i\}$ be an interior point, and define
\beqo
r_1:=\sup \{r>0:  B_r(x_1)\subset \{u(x)=a_i\}\}.
\eeqo
When $u\not\equiv a_i$, we know that $r_1\in(0,\infty)$ and there is a point $x_2\in \pa B_{r_1}(x_1)$ such that $u(x_2)=a_i$ and $x_2\in \overline{\{|u-a_i|>0\}}$. For $r$ sufficiently small, by Lemma \ref{nondeg} we have
\beq\label{nondeg optimal reg}
\sup\limits_{B_r(x_2)}|u-a_i|\geq cr^{1+\beta}.
\eeq
Also we claim that $\na u(x_2)=0$. Firstly, by $u(x)\equiv a_i$ for any $x\in B_r(x_2)\cap B_{r_1}(x_1)$, we can easily check that $\pa_\nu u(x_2)=0$ where $\nu$ denotes the normal vector on $\pa B_{r_1}(x_1)$. On the other hand, take any unit vector $\mu$ which points to one of the tangential directions at $x_2$ on $\pa B_{r_1}(x_1)$. For any $|t|\leq r_1$,
\beqo
x_2+t\mu+(\sqrt{r_1^2-t^2}-r_1)\nu\in\pa B_{r_1}(x_1),\quad u(x_2+t\mu+(\sqrt{r_1^2-t^2}-r_1)\nu)=a_i.
\eeqo
Thus for any $i\in\{1,2,...,m\}$ we have
\beqo
0=\f{d}{dt}\bigg|_{t=0}u_i(x_2+t\mu+(\sqrt{r_1^2-t^2}-r_1)\nu)=\na u_i(x_2)\cdot \mu.
\eeqo
Since $\mu$ is an arbitrary tangential direction, the claim $\na u(x_2)=0$ is proved.

Finally, the vanishing of $\na u(x_2)$ together with \eqref{nondeg optimal reg} imply the sharpness of the $C^{1,\beta}$ regularity.

\end{proof}


\begin{rmk}\label{c1,alpha remark}
For $1\leq \al<2$, by exactly the same argument we can also prove $u\in C^{1,\g}_{loc}$ for any $\g\in(0,1)$. Just notice that we need to replace \eqref{diff w alpha} with $W(u)-W(v)\leq C|u-v|$. When $\al=1$, in contrast to the scalar case problem (see \cite{sha}), it is still open whether $u\in C^{1,1}_{loc}$. 
\end{rmk}

Now with the regularity result we can identify the Euler-Lagrange equation for the entire minimizer $u$ in the following lemma.

\begin{lemma}\label{EL}
Let $u$ be an entire minimizer of the functional $J$ satisfying $|u(x)|\leq M$. Assume H1, H2. Then we have

\begin{enumerate}
\itemsep0.5em
\item $1<\al<2$, $u$ is a strong solution of
\beq\label{eleq for 1<al<2}
\D u=W_u(u),\qquad \forall x\in\BR^n.
\eeq
\item $\al=1$, $u$ is a strong solution of
\beq\label{eleq for al=1}
\D u=W_u(u)\chi_{\{d(u,A)>0\}}.
\eeq
\item $0<\al<1$, in the open set $\{x: d(u(x),A)>0\}$, $u$ solves the equation \eqref{eleq for 1<al<2}.
\end{enumerate}
Here $W_u(u)$ denotes the derivative of $W$ with respect to $u$, $\chi$ is the characteristic function and $d(x,A):=\dist(x,A)$.

\end{lemma}

\begin{proof}
Take $D\in\BR^n$ to be an arbitrary bounded Lipschitz domain, and for $\phi\in C_0^\infty(D,\BR^m)$, we compute the first variation of the energy $J_D(u)$
\beq\label{first variation}
\begin{split}
0&\leq \int_{D} \left(\f12|\na (u+t\phi)|^2-\f12|\na u|^2+W(u+t\phi)-W(u)\right)\,dx\\
& = t\int_{D}\na u\cdot \na \phi\,dx+\f{t^2}{2}\int_{D}|\na\phi|^2\,dx+\int_{D}\left(W(u+t\phi)-W(u)\right)\,dx.
\end{split}
\eeq

When $1<\al<2$, $W$ can be written as \eqref{formula for W} with the condition \eqref{condition on g}. One can directly compute
\beq\label{formula W_u(u)}
W_u(u)=\left( \al(u-a_i)|u-a_i|^{\al-2}\prod\limits_{k\neq i}|u-a_k|^\al \right)g(u)+\left(\prod\limits_{i=1}^N |u-a_i|^\al \right)D_ug(u).
\eeq
Since we already proved $u\in C^{1,\g}(D)$, it is obvious that $W_u(u)\in C(D)$. Therefore when we divide \eqref{first variation} by $t$ and take $t\ri 0+$ or $0-$, it follows that $u$ solves \eqref{eleq for 1<al<2} in $D$.

For $\al=1$, dividing \eqref{first variation} by $t$ and letting $t\ri 0$, we can prove that
\beqo
\left|\int_{D}\na u\cdot \na \phi\,dx\right|\leq C\int_{D}|\phi|\,dx,
\eeqo
which implies $\D u\in L^\infty(D,\BR^m)$. Moreover, on any sub-domain $K\subset D\cap \{d(u,A)>0\}$, the equation \eqref{eleq for 1<al<2} holds. Combining the fact that $\D u=0$ a.e. on $\{x:u(x)\in A\}$ (which is due to the fact that weak derivatives for Sobolev functions vanish on level sets a.e. \cite{book-eg}), we have (since $u$ is continuous) that $u$ is a strong solution of \eqref{eleq for al=1} in $D$.

For $0<\al<1$, one can easily check that $u$ solves \eqref{eleq for 1<al<2} in the open set $\{x: d(u,A)>0\}$. However, since $W_u(u)$ blows up as $d(u,A)\ri 0$, the local integrability of $W_u(u)$ is not a priori known, so $u$ may not solve \eqref{eleq for al=1} on the whole space in distributional sense.
\end{proof}

\begin{rmk}
In the case $0<\al<1$, even though we cannot say $u$ is a distributional solution to \eqref{eleq for al=1}, we can deduce another equation for $u$ from the first domain variation. For any $\varphi\in C_0^1(D,\BR^n)$, it holds that
\begin{align*}
0&=\f{d}{dt}J_{D}(u(x+t\varphi(x)))\\
&=\int_D \left((\na u\na\va)\cdot \na u-(\di\va)(\f12|\na u|^2+W(u))\right)\,dx
\end{align*}
This formulation has also been utilized in \cite{lp,Weiss3}. Here we just present this form of equation for completeness and it will not be utilized in the rest of the paper.
\end{rmk}

With the Euler-Lagrange equation \eqref{eleq for 1<al<2} and the formula for $W_u(u)$ \eqref{formula W_u(u)}, we can easily improve the regularity of $u$ when $1<\al<2$.

\begin{proposition}\label{reg al>1}
When $1<\al<2$, the entire minimizer $u\in C^{2,\al-1}_{loc}(\BR^n,\BR^m)$.
\end{proposition}

\begin{proof}
According to Lemma \ref{EL}, $u$ satisfies the equation \eqref{eleq for 1<al<2} when $1<\al<2$. Using the formula \eqref{formula W_u(u)} for $W_u(u)$ and the rough estimate $u\in C^{1,\gamma}_{loc}$, we see that $W_u(u)\in C^{\al-1}_{loc}(\BR^n,\BR^m)$. Then the $C^{2,\al-1}_{loc}$ regularity immediately follows from the classical Schauder estimate.
\end{proof}

\section{Estimate of $\cl^n(I_0)$ and existence of the free boundary}\label{section 3}
In this section we will prove the estimate \eqref{measure of I_0} for any nontrivial entire minimizer $u$ of the functional $J$. Furthermore, we also show that at every sufficiently large scale, the minimizer $u$ must contain at least two different phases.

Take $W$ satisfying the hypothesises H1 $\&$ H2 and assume $u$ is an entire minimizer for the functional $J$. Before stating our new results, we first recall two estimates from \cite{book-afs} and \cite{agz} without proof, which will play an important role in our arguments. Readers can refer to \cite{book-afs} and \cite{agz} for detailed proofs. These estimates for the scalar case $m=1$ are obtained in \cite{cc}.

\begin{proposition}\label{two estimates}
When $0<\al<2$, for any entire minimizer $u$ satisfying $\|u\|_{L^\infty(\BR^n)}<\infty$, the following two estimates hold true:
\begin{enumerate}
\itemsep0.5em
\item \underline{The basic estimate} (see \cite[Lemma 2.2]{agz})
For any $x_0\in\BR^n$, there exists an $r_0$ such that for $r>r_0$,
\beq\label{basic est}
J_{B_r(x_0)}(u)\leq Cr^{n-1},
\eeq
where the constant $C=C(M)$ is independent of $u$. We notice that this $r_0$ can be $0$ when $\al\in[1,2)$.
\item \underline{The density estimate} (see \cite[Theorem 5.2]{book-afs}) Take $a\in A$ to be a minimal point for $W(u)$. If for some $r_0,\,\lam,\,\mu_0>0$,
    \beqo
    \mathcal{L}^n(B_{r_0}(x)\cap \{|u-a|>\lam\})\geq \mu_0,
    \eeqo
    then there exists a constant $C(\mu_0,\lam)>0$ such that
    \beq\label{density est}
    \mathcal{L}^n(B_{r}(x)\cap \{|u-a|>\lam\})\geq C(\mu_0,\lam) r^n, \quad \forall r\geq r_0.
    \eeq
\end{enumerate}
\end{proposition}

Another important component of our arguments is the following non-degeneracy lemma.

\begin{lemma}\label{nondeg}
Assume $0<\al<2$. We take the point $a_1\in A$ and an entire minimizer $u$ for the functional J such that $\|u\|_{L^\infty(\BR^n)}<\infty$. There exists a suitably small number $\theta(W)$ and a constant $c=c(n,W)$, such that if $x_0\in\overline{\{0<|u-a_1|<\theta\}}$ and $B_r(x_0)\subset\overline{\{|u-a_1|<\theta\}}$, then
\beq\label{nondegeneracy}
\sup\limits_{B_r(x_0)}|u-a_1|\geq c(n,W)r^\f{2}{2-\al},
\eeq
where the constant $c(n,W)$ only depends on the dimension $n$ and the potential function $W$. Moreover $c(n,W)\sim O(\alpha)$ for $\alpha<<1$.
\end{lemma}

\begin{proof}
Without loss of generality, suppose $a_1=0^m$. First we require that
\beqo
\theta<\f12\min\limits_{i\neq j}|a_i-a_j|=r_0.
\eeqo
By (H2), when $|u|<r_0$, $W(u)$ can be written as $W(u)=|u|^\al\cdot g(u)$ for some $g(u)\in C^2(B_{r_0}(0^m))$ satisfying \beqo
g(u)\geq C_g
\eeqo
for some constant $C_g>0$.

Assume $|u(x_0)|>0$ (if $|u(x_0)|=0$, then we simply take a sequence of points $\{x_i\}$ converging to $x_0$ and satisfy $|u(x_i)|>0$). Taking $h(x)=|u|^{2-\al}-c|x-x_0|^2$ for some constant $c$ which will be determined later, by direct calculation we have that if $|u(x)|>0$, then
\beq\label{laplaceh}
\D h=(2-\al)(-\al)\f{\left|\na|u|\right|^2}{|u|^\al}+(2-\al)\f{|\na u|^2}{|u|^\al}+(2-\al)u\cdot D_u(g)+\al(2-\al)g(u)-2nc
\eeq
If we take
\beq\label{condition on theta, c}
\theta<\min\{\f{\al C_g}{4\|D_ug(u)\|_{L^\infty(B_{r_0}(0^m))}}, \f12\min\limits_{i\neq j}|a_i-a_j| \},\quad c<\f{\al(2-\al)C_g}{8n},
\eeq
then \eqref{laplaceh} implies that
\beqo
\Delta h\geq -(2-\al)\al \f{\left|\na|u|\right|^2}{|u|^\al}+(2-\al)\f{\left|\na u\right|^2}{|u|^\al}
\eeqo

When $\al\leq 1$, it follows that $\Delta h\geq 0$ in $\{|u(x)|>0\}\cap B_r(x_0)$. Since $h(x_0)>0$ and $h(x)<0$ on $\pa\{|u(x)>0|\}\cap B_r(x_0)$, we must have
\beqo
\max\limits_{x\in\pa B_r(x_0)}h(x)>0,
\eeqo
which implies the lemma.

For $\al\in(1,2)$, combining \eqref{laplaceh} and \eqref{condition on theta, c} we deduce that in $\{|u|>0\}\cap B_{r}(x_0)$
\begin{align*}
&\D h+(2-\al)\al\f{\left|\na|u|\right|^2}{|u|^\al}-(2-\al)\f{\left|\na u\right|^2}{|u|^\al}\geq \f{\al(2-\al)C_g}{4}\\
\Rightarrow & \D h+\f{\al}{2-\al}\f{\na h\cdot \na(|u|^{2-\al}+c|x-x_0|^2)+4c^2|x-x_0|^2}{|u|^{2-\al}}\geq (2-\al)\f{\left|\na u\right|^2}{|u|^\al}+\f{\al(2-\al)C_g}{4}\\
\Rightarrow & \D h+\left( \f{\al\na(h+2c|x-x_0|^2)}{(2-\al)|u|^{2-\al}}  \right)\cdot\na h+\f{4c\al(-h+|u|^{2-\al})}{(2-\al)|u|^{2-\al}}\geq \f{\al(2-\al)C_g}{4}\\
\Rightarrow & \D h+\left( \f{\al\na(h+2c|x-x_0|^2)}{(2-\al)|u|^{2-\al}}  \right)\cdot\na h-\f{4c\al}{(2-\al)|u|^{2-\al}}\cdot h\geq 0.
\end{align*}
Here to derive the last inequality we further require that $c$ satisfies
\beq\label{condition on c}
c\leq \f{(2-\al)^2C_g}{16}.
\eeq
Then the maximum principle argument can be applied again to get
\beqo
\max\limits_{x\in\pa B_r(x_0)} |u(x)|> cr^\f{2}{2-\al}.
\eeqo
This completes the proof.
\end{proof}
\vspace{5mm}

Now we are ready to prove the first part \eqref{measure of I_0} of the Theorem \ref{main}. For the sake of convenience, we rewrite the statement in the following theorem.

\begin{theorem}[First part of Theorem \ref{main}]\label{first part main theorem}
Let $x_0\in\BR^n$, $u: \BR^n\ri \BR^m$ be a bounded nonconstant entire minimizer of the energy $J$. Then there are positive constants $R_0$ and $c$ such that
\beq\label{I_0 est}
\cl^n(B_R(x_0)\cap I_0)\leq cR^{n-1},\quad R>R_0,
\eeq
where $I_0$ is defined in \eqref{I_0}, which is the region where $W(u)>0$. The constant $c$ only depends on the dimension $n$, the potential function $W$ and $\|u\|_{L^\infty(\BR^n)}$.
\end{theorem}

\begin{proof}
Without loss of generality, suppose $x_0=0^n$ and write $B_R=B_R(0^n)$. According to the basic estimate \eqref{basic est} in Proposition \ref{two estimates}, we know that there exist positive constants $C_0, \,r_0$ such that for any $R>r_0$
\begin{equation}
\label{energy density 0}\int_{B_R} \f12|\na u|^2+W(u)\,dx\leq C_0R^{n-1}.
\end{equation}

For the sake of convenience, we use the cubes which are centered at $0^n$ to replace $B_R$. Define
\beqo
\tilde{S}_R:=\{x\in \BR^n: x_i\in(-R,R),\text{ for }i=1,2,...,n\}.
\eeqo
Let $L$ be a constant whose value will be specified later. For any positive integer $k$,  we can divide the cube $\tilde{S}_{kL}$ into $(2k)^n$ identical cubes with the side length $L$. We number all these sub-cubes by $S_1,S_2,...,S_{K}$ where $K:=(2k)^n$. And we take $\theta$ to be the constant $\theta(W)$ in the Lemma \ref{nondeg}. Then we define
\beq\label{def sigma i,j}
\sigma_i^j:=\cl^n(\{|u-a_j|<\f{\th}{2}\}\cap S_i), \quad \text{for }i=1,...K,\; j=1,..., N.
\eeq

Take $\e:=\e(\theta)$ to be a small constant to be specified later, depending only on $\theta$ and $\|u\|_{L^\infty}$. Also we introduce the notion of adjacent sub-cubes: $S_{i_1}$ and $S_{i_2}$ are called adjacent if and only if
\beqo
\overline{S_{i_1}}\cap \overline{S_{i_2}}\neq \varnothing,\quad  1\leq i_1,\,i_2\leq K
\eeqo
We divide $\{S_i\}_1^K$ into the following five non-overlapping classes
\begin{itemize}
\itemsep0.5em
\item[1] Boundary sub-cubes of $\tilde{S}_{kL}$:
\beqo
T_1:=\{S_i:\text{ the number of adjacent cubes of }S_i \text{ is less than }3^n-1\}.
\eeqo
\item[2] Sub-cubes that contain two phases:
\beqo
T_2:=\{S_i: \exists j_0 \text{ s.t. } \s_i^{j_0}=\max\limits_{1\leq j\leq N}\s_i^j\leq(1-2\e)L^n, \;\max\limits_{j\neq j_0} \s_i^j\geq\f{\e}{N-1} L^n\}\backslash T_1.
\eeqo
\item[3] Sub-cubes that contain regions where $u$ stays away from any $a_j$.
\beqo
T_3:=\{S_i: \exists j_0  \text{ s.t. } \s_i^{j_0}=\max\limits_{1\leq j\leq N}\s_i^j\leq(1-2\e)L^n, \;\max\limits_{j\neq j_0} \s_i^j<\f{\e}{N-1} L^n\}\backslash T_1.
\eeqo
\item[4] ``Interior" sub-cubes of the contact set $\{x:u(x)\in A\}$:
    \beqo
    T_4:=\{S_i: \exists j_0 \text{ s.t. } \s_i^{j_0}>(1-2\e)L^n \text{ and }\s_p^{j_0}> (1-2\e)L^n,
     \forall S_p\text{ adjacent to }S_{i}\}\backslash T_1.
     \eeqo
\item[5] Sub-cubes close to the boundary of the contact set:
\beqo
T_5:=\{S_i: \exists j_0 \text{ s.t. } \s_i^{j_0}> (1-2\e)L^n \text{ and }\s_p^{j_0}\leq  (1-2\e)L^n,\text{ for some } S_p\text{ adjacent to }S_{i}\}\backslash T_1.
\eeqo
\end{itemize}

Now we estimate the number of cubes in each class. First note that $S_i\in T_1$ means $S_i$ is one of the boundary cubes of $\tilde{S}_{kL}$, therefore
\beq\label{number of T_1}
|T_1|\leq c_0(n)k^{n-1} \quad \text{ for some dimensional constant }c_0(n).
\eeq
For a cube $S_i\in T_2$, assume $\s_i^{j_0}=\max\limits_{1\leq j\leq N}\s_i^j\leq (1-2\e)L^n$ and $\s_i^{j_1}\geq \f{\e}{N-1}L^n$ for some $j_1\neq j_0$. By the definition of $\s_i^j$ and $\theta<\f12|a_{j_0}-a_{j_1}|$, we can infer that for any $r\in[\f{\th}{2},\,|a_{j_0}-a_{j_1}|-\f{\th}{2}]$, it holds that
\beqo
\cl^n(\{|u-a_{j_0}|<r\}\cap S_i)\geq  \f{\e}{N-1} L^n,\quad \cl^n(S_i\backslash \{|u-a_{j_0}|<r\})\geq \f{\e}{N-1} L^n.
\eeqo
Applying the co-area formula and the relative isoperimetric inequality (see for example \cite{Thomas}), we have
\begin{equation}\label{co-area0}
\begin{split}
&\int_{S_i}|\na u|^2\,dx\\
\geq & \f{1}{L^n} \left(\int_{S_i} |\na(u-a_{j_0})|\,dx\right)^2\\
\geq &\f{1}{L^n} \left( \int_{\th/2}^{|a_{j_0}-a_{j_1}|-\th/2} \ch^{n-1}(\{|u-a_{j_0}|=r\}\cap S_i)\,dr \right)^2\\
\geq & \f{1}{L^n} \left( \int_{\th/2}^{|a_{j_0}-a_{j_1}|-\th/2} C \left(\min\{\cl^n(\{|u-a_{j_0}|<r\}\cap S_i),\cl^n(S_i\backslash \{|u-a_{j_0}|<r\}) \}\right)^{\f{n-1}{n}} \,dr \right)^2\\
\geq &c_1(L,\th,\e)>0
\end{split}
\end{equation}
From the basic estimate \eqref{basic est} we get
\beq\label{number of T_2}
|T_2|\leq \f{C(kL)^{n-1}}{c_1}= c_2(L,\th,\e)k^{n-1},\quad k\geq k_0,
\eeq
where $k_0$ is a constant.

For a cube $S_i\in T_3$,
\beqo
\cl^n\left(\left\{|u-a_j|>\f{\theta}{2},\;\forall 1\leq j\leq N\right\}\cap S_i\right)>\e L^n
\eeqo
By the Hypothesis H1 and H2 on $W$ and the assumption $\|u\|_{L^\infty}<\infty$, there is a constant $c_3$ which depends on $\|u\|_{L^\infty},\theta$ such that
\beqo
W(u)\geq c_3,\quad \text{when }|u-a_j|>\f{\theta}{2},\; \forall 1\leq j\leq N
\eeqo
Thus by \eqref{energy density 0} the number of sub-cubes $T_3$ is bounded by
\beq\label{number of T_3}
|T_3|\leq c_4(L,\th,\e,\|u\|_{L^\infty})k^{n-1},\quad k\geq k_0.
\eeq

From now on we focus on the analysis of cubes in $T_4$ and $T_5$. Take $S_i$ in $T_4$ or $T_5$, then there is a $j_0$ such that $\s_i^{j_0}>(1-2\e)L^n$. In this case, we claim that when $\e$ is suitably chosen, we can assure that
\beqo
\max\limits_{x\in S_i}|u(x)-a_{j_0}|<\theta
\eeqo
If there exists $x_0\in S_i$ such that $|u(x_0)-a_{j_0}|\geq \theta$, then we have that there exists a constant $c_5(\theta,\|\na u\|_{L^\infty})$ such that
\beqo
\cl^n(\{|u-a_{j_0}|>\theta/2\}\cap  S_i)\geq c_5.
\eeqo
We note that the uniform boundedness of $|\na u|$ follows from the $C^{1,\beta}$ regularity (Proposition \ref{reg 0<al<=1} and Proposition \ref{reg al>1}) and the assumption that $|u|$ is uniformly bounded. And $c_5$ doesn't depend on $j_0$. Then the claim follows if we simply take
\beq\label{value of epsilon}
\e<\f{c_5}{2L^n}.
\eeq

\begin{lemma}\label{value of L}
When $L$ is suitably chosen depending on $\theta$, in any cube $S_i\in T_4\cup T_5$, it holds
\beq\label{measureofa_j}
\cl^n(\{u(x)=a_{j_0}\}\cap S_i)\geq \om_n\left(\f{L}{4}\right)^n
\eeq
where $\om_n$ is the volume of the n-dimensional unit ball.
\end{lemma}

\begin{proof}
We proceed by contradiction and denote the central point of $S_i$ by $z_i$. So,
\beqo
|\{u(x)=a_{j_0}\}\cap S_i|< \om_n\left(\f{L}{4}\right)^n
\eeqo
Then there must be a point $x_1\in B_R(z_i,\f{L}{4})$ such that $x_1\in \overline{\{0<|u-a_{j_0}|<\theta\}}$. Moreover, we have
\beqo
B_{\f{L}{4}}(x_1)\subset S_i\subset \overline{\{|u-a_{j_0}|<\theta\}}
\eeqo
Therefore we are in the position to apply Lemma \ref{nondeg} to deduce that
\beqo
\sup\limits_{B_{\f{L}{4}}(x_1)}|u-a_{j_0}|\geq c(n,W)\left(\f{L}{4}\right)^{\f{2}{2-\al}}
\eeqo
which contradicts with $\max\limits_{x\in S_i}|u-a_{j_0}|<\theta$ if we choose the constant $L$ at the beginning satisfying $c(n,W)\left(\f{L}{4}\right)^{\f{2}{2-\al}}>2\theta$. This completes the proof of Lemma \ref{value of L}.
\end{proof}

If the cube $S_i\in T_4$, then by definition we have
\beqo
|u(x)-a_{j_0}|<\theta, \quad \forall x\in S_i\cup(\bigcup_{S_p\text{ adjacent to }S_i} S_p)
\eeqo
By the same argument as in the proof of the lemma above, we obtain that
\beqo
u(x)\equiv a_{j_0},\quad x\in S_i.
\eeqo

If $S_i\in T_5$, then there must be at least one adjacent cube of $S_i$, denoted by $S_{p}$, such that
\beq\label{est in Sp}
|\{|u-a_{j_0}|>\f{\theta}{2}\}\cap S_{p}|>\e L^n.
\eeq
We set
\beqo
Q_{S_i}:=S_i\cup(\bigcup_{S_p\text{ adjacent to }S_i} S_p)
\eeqo
Then by \eqref{measureofa_j}, \eqref{est in Sp} and the co-area formula, we can compute similarly as in \eqref{co-area0} to get
\beqo
\int_{Q_{S_i}}|\na u|^2\,dx\geq c_6(L, \th, \e).
\eeqo
Since each point can belong to at most $3^n$ different $Q_{S_i}$, utilizing \eqref{basic est} we conclude
\beqo
C(n)(kL)^{n-1}\geq \sum\limits_{S_i\in T_5}\int_{Q_{S_i}}|\na u|^2\,dx\geq c_6|T_5|,
\eeqo
which implies
\beq\label{number of T_5}
|T_5|\leq c_7(n, L,\th,\e)k^{n-1}.
\eeq

Finally, combining \eqref{number of T_1}, \eqref{number of T_2}, \eqref{number of T_3} and \eqref{number of T_5} we get
\beq
\cl^n(\tilde{S}_{kL}\cap I_0)\leq (|T_1|+|T_2|+|T_3|+|T_5|)L^n\leq c_8(n, L,\th,\e)(kL)^{n-1}.
\eeq
Since $B_{kL}\subset \tilde{S}_{kL}$, we can get \eqref{I_0 est} after taking $k$ to be the smallest integer larger than $\f{R}{L}$ for sufficiently large $R$. Also if we carefully check the definitions of all the constants in the proof we conclude that $c_8$ only depends on the dimension $n$, the potential $W$ and the uniform bound of $|u|$, but not on the specific solution $u$. This completes the proof of Theorem \ref{first part main theorem}.

\end{proof}

Theorem \ref{first part main theorem} implies that a bounded entire minimizer $u(x)$ should satisfy $W(u)=0$ in ``most of the space". Next we further show that at sufficiently large scales, $u$ must possess at least two different phases, each of which contains some definite measure of order $R^n$.

\begin{lemma}\label{lemma ci(th)}
Let $x_0\in\BR^n$, $u:\BR^n\ri\BR^m$ be a bounded entire minimizer of $J$. Assume that $u\not\equiv a_i$ for any $i\in \{1,2,...,N\}$. We take an arbitrary constant $\theta<r_0:=\f12 \min\limits_{1\leq i\neq j\leq N}|a_i-a_j|$, then there exist positive constants $R_0, c(u,\theta)$ such that for any $R\geq R_0$, there are $a_i,a_j\in A$, which depend on $R$, satisfying
\beq\label{ci(th)}
\cl^n(B_R(x_0)\cap \{|u-a_k|<\theta\})\geq cR^n,\quad k=i,j.
\eeq
\end{lemma}

\begin{proof}
Since $u$ is nonconstant, by $C^{1,\beta}$ regularity of $u$ there is some $R_1>0, \,0<\lam<r_0,\,\mu_0>0$ such that
\beqo
\cl^n(B_{R_1}(x_0)\cap \{|u-a_1|>\lam\})\geq \mu_0.
\eeqo
Then by the density estimate \eqref{density est} in Proposition \ref{two estimates}, there exists $\mu_1$ such that
\beq\label{u-a1>lam}
\cl^n(B_{R}(x_0)\cap \{|u-a_1|>\lam\})\geq \mu_1R^n,\quad \forall R\geq R_1
\eeq
Take $\theta<r_0$ to be an arbitrary constant. By our hypothesis on $W$, there is a positive constant $C=C(\lam,\theta, \|u\|_{L^\infty})$ such that
\beqo
W(u)>C, \;\text{ when } |u-a_1|>\lam,\,|u-a_j|\geq \theta \text{ for any }j\neq 1.
\eeqo
Applying the basic estimate \eqref{basic est} in Proposition \ref{two estimates}, for enough large $R$,
\beq\label{u-a_2>theta}
\cl^n(B_R(x_0)\cap \{|u-a_1|>\lam,\,|u-a_j|\geq \theta \text{ for any }j\neq 1\})\leq C_2R^{n-1},
\eeq
for some constant $C_2$. Combining \eqref{u-a1>lam} and \eqref{u-a_2>theta}, we obtain that
\beqo
\begin{split}
&\cl^n(B_{R}(x_0)\cap (\bigcup\limits_{j\neq 1}\{|u-a_j|<\theta\}))\\
\geq & \cl^n(B_R(x_0)\cap \{|u-a_1|>\lam, |u-a_j|<\theta \text{ for some }j\neq 1\})\geq c_1(u,\theta)R^n, \quad \forall R>\tilde{R}_1
\end{split}
\eeqo
for some constants $\tilde{R}_1$ and $c_1$. The same argument also works for the set $B_R(x_0)\cap(\bigcup\limits_{j\neq k}\{|u-a_j|<\theta\})$ for any $k\in \{1,2,...,N\}$, i.e. there exists $\tilde{R}_k, \;c_k>0$ such that
\beqo
\cl^n(B_{R}(x_0)\cap (\bigcup\limits_{j\neq k}\{|u-a_j|<\theta\}))\geq c_k(u,\theta) R^n, \forall R\geq \tilde{R}_k
\eeqo
Finally, we take $R_0=\min\limits_{k} \tilde{R}_k$ and $c=\f{1}{N-1}\min\limits_k c_k$ and the conclusion of the lemma easily follows.
\end{proof}

In the following theorem, we show that in any ball $B_R(x_0)$ with radius $R$ large enough, the sets $\{u=a_i\}$ and $\{u=a_j\}$ ($a_i,\,a_j$ from Lemma \ref{lemma ci(th)}) must contain a set of measure of the order $R^n$.

\begin{theorem}\label{two phase existence}
Let $x_0\in \BR^n$, $u: \BR^n\ri\BR^m$ be a bounded entire minimizer of the energy $J$, and $u\not\equiv a_j$ for any $j\in\{1,2,...,N\}$. Then there are positive constants $R_0$ and $c$ (both depend on $u$) such that for any $R\geq R_0$, there are $a_i,a_j$ depending on $R$ such that
\beq
\min\{\mathcal{L}^n(B_R(x_0)\cap \{u=a_i\}),\mathcal{L}^n(B_R(x_0)\cap \{u=a_j\})\}\geq  cR^n,\quad \forall R>R_0
\eeq
\end{theorem}
\begin{proof}
Without loss of generality, suppose $x_0=0^n$ and write $B_R=B_R(0^n)$. According to the Proposition \ref{two estimates} and Lemma \ref{lemma ci(th)}, we know that for any sufficiently large $R$, there are $a_i,a_j\in A$ such that \eqref{ci(th)} holds.

The proof relies on the same technique as the proof of Theorem \ref{first part main theorem}. So we will only present the main ingredients and omit some technical details. Take $L$ as the same constant in Theorem \ref{first part main theorem} and $k\in\mathbb{N}$. We consider the domain
\beqo
\tilde{S}_{kL}:=\{x\in \BR^n: x_i\in (-kL,kL)\},
\eeqo
and then divide $\tilde{S}_{kL}$ into $K=(2k)^n$ identical sub-cubes $S_1,...,S_K$, each of which has side of length $L$. We also recall the definition of $\s_i^j$ in \eqref{def sigma i,j}. By Lemma \ref{lemma ci(th)}, there are two phases $a_i, a_j$ (for simplicity we assume they are $a_1,a_2$) such that
\beq\label{a_1,a_2 in S_kL}
\cl^n(\tilde{S}_{kL}\cap \{|u-a_j|<\f{\th}{2}\})\geq c(kL)^n,\quad j=1,2.
\eeq

Take $\e:=\e(u,\theta)$ be a small constant such that
\begin{itemize}
\itemsep0.5em
\item[a.] \eqref{value of epsilon} holds. As a result, if $\s_i^j>(1-2\e)L^n$, then $|u(x)-a_j|<\theta$ for any $x\in S_i$.
\item[b.] $\e\leq \f{c}{2^{n+3}}$ where $c$ is the constant in \eqref{a_1,a_2 in S_kL}.
\end{itemize}

Then we divide $\{S_i\}_1^K$ into the following two classes
\begin{itemize}
\itemsep0.5em
\item[1] $U_1:=\{S_i: \exists j_0 \text{ s.t. } \s_i^{j_0}=\max\limits_{1\leq j\leq N}S_i^j\leq(1-2\e)L^n\}.$
\item[2] $U_2:=\{S_i: \exists j_0 \text{ s.t. } \s_i^{j_0}=\max\limits_{1\leq j\leq N}S_i^j>(1-2\e)L^n.\}$
\end{itemize}

From the proof of Theorem \ref{first part main theorem}, we have
\beqo
|U_1|\leq c_0(L,\theta,\e)k^{n-1}.
\eeqo

Let $K_1$ denote the number of sub-cubes $S_i$ satisfying $\s_i^1>(1-2\e)L^n$. We obtain from \eqref{a_1,a_2 in S_kL}
\beq
\begin{split}
c(kL)^n&\leq \sum\limits_{1\leq i\leq (2k)^n}\s_i^1\\
&\leq |U_1|L^n+K_1 L^n+ \left((2k)^n-|U_1|-K_1 \right)(2\e L^n)\\
&\leq c_0k^{n-1}L^n +K_1L^n+\f{c}{4} (kL)^n\quad (\text{Property b of }\e),
\end{split}
\eeq
which immediately implies that $K_1\geq \f{c}{2}k^n$ whenever $k$ is large enough. Together with Lemma \ref{value of L} we have
\beq\label{estimate two phase}
\cl^n (\tilde{S}_{kL}\cap \{u=a_1\})\geq \f{c}{2}k^n\om_n (\f{L}{4})^n\geq c_1(kL)^n,
\eeq
for some constant $c_1=c_1(W,u)$. For $\{u=a_2\}$ the estimate \eqref{estimate two phase} still holds. One can easily check that \eqref{estimate two phase} implies the statement of Theorem \ref{two phase existence}.

\end{proof}


\section{Weiss' Monotonicity formula and a growth estimate in the case $\al=1$}\label{growth est}

Thanks to Theorem \ref{two phase existence}, we know for a uniformly bounded entire minimizer $u$, the free boundary $\pa\{|u-a_i|>0\}$ $(i=1,2)$ must exist. In this section we will derive a growth rate estimate for $|u-a_i|$ away from the free boundary in the case $\underline{\al=1}$.

From now on we fix $\al=1$ and assume $a_1=0^m$. By the hypothesis H2, $W(u)$ has the form $W(u)=g(u)|u|$ for some $g\in C^2(B_\theta)$. Here $\th$ is the constant in Lemma \ref{nondeg}. Since $u$ is a local minimizer, it satisfies the Euler-Lagrange equation near the free boundary point,

\beq\label{ELeq}
\D u= g(u)\f{u}{|u|}+|u| D_u g(u).
\eeq

Also there exists a positive constant $C>0$ such that $g(u)>C$ when $|u|\leq \theta$. We use the notation
\beq\label{notations}
\Om(u):=\{|u(x)|>0\},\quad \G(u):=\pa^*\Om(u).
\eeq
Here $\pa^*$ denotes De Giorgi's reduced boundary. An easy observation is that for any point $x\in\G(u)$, we must have $|u(x)|=|\na u(x)|=0$. The proof is straightforward: if at some point $x_0\in \G(u)$, $|\na u|>0$, then by continuity of $\na u$ we have that in a small neighborhood $B_{r}(x_0)$, $|\na u_i|\geq c$ for some $1\leq i\leq m$ and $c>0$. The inverse function theorem implies that in $B_r(x_0)$, $\{u_i=0\}$ is a $(n-1)$-dimensional hypersurface, which further gives $x_0\not\in\pa^e\Om(u)$, where $\pa^e$ denotes the measure theoretic boundary. Finally we arrive at a contradiction thanks to the well-known result $\pa^* E\subset \pa^e E$ for any set $E$ of locally finite perimeter. For the definitions of the reduced boundary and the measure theoretic boundary, as well as their relationship, we refer to \cite[Chapter 5.7\&5.8]{book-eg} for details.

We first establish an almost monotonicity formula for $|u|<\th$. The proof closely follows the classical arguments of Weiss (see \cite{Weiss1,Weiss2})

\begin{lemma}\label{weiss lemma}
Let $u$ be a solution of \eqref{ELeq} in $B_r(x_0)$ such that $|u|<\th$ in $B_r(x_0)$, and set
\beq\label{def weiss}
W(u,x_0,r)=\f{1}{r^{n+2}}\int_{B_r(x_0)}\left(\f12|\na u|^2+g(u)|u|\right)\,dx-\f{1}{r^{n+3}}\int_{\pa B_r(x_0)}|u|^2 \,d\mathcal{H}^{n-1}.
\eeq

Then $W(u,x_0,r)$ satisfies
\beq\label{Weiss}
\f{d}{dr} W(u,x_0,r) =r\int_{\pa B_1}|\f{du_r}{dr}|^2\,d\mathcal{H}^{n-1}+2r\int_{B_1}D_ug\cdot u_r|u_r|\,dx
\eeq
where
\beqo
u_r(x):=\f{u(x_0+rx)}{r^2}.
\eeqo
\end{lemma}

\begin{proof}
First we write $W(u,x_0,r)$ as
\beqo
W(u,x_0,r)=\int_{B_1}\left(\f12|\na u_r|^2+g(r^2 u_r)|u_r|\right)\,dx-\int_{\pa B_1}|u_r|^2\,d\mathcal{H}^{n-1}.
\eeqo
Then by direct calculation we have
\begin{align*}\
&\f{d}{dr}W(u,x_0,r)\\
=&\int_{B_1}\left( \na u_r\cdot\f{d}{dr}(\na u_r)+D_u g(r^2 u_r)\cdot \f{d}{dr}(r^2 u_r)|u_r|+ g(r^2 u_r)|u_r|^{-1}u_r\cdot \f{d}{dr}u_r \right)\,dx\\
& -2\int_{\pa B_1}u_r\cdot\f{d}{dr}u_r\,d\mathcal{H}^{n-1}\\
=&\int_{B_1}\left(-\Delta u_r\cdot \f{d}{dr}u_r+D_u g(r^2 u_r)\cdot \f{d}{dr}(r^2 u_r)|u_r|+ g(r^2 u_r)|u_r|^{-1}u_r\cdot \f{d}{dr}u_r \right)\,dx\\
& -2\int_{\pa B_1}u_r\cdot\f{d}{dr}u_r\,d\mathcal{H}^{n-1}+\int_{\pa B_1}(x\cdot \na u_r)\cdot \f{d}{dr}u_r\,d\mathcal{H}^{n-1}.\\
=&\int_{B_1}\left(-\Delta u_r\cdot \f{d}{dr}u_r+D_u g(r^2 u_r)\cdot \f{d}{dr}(r^2 u_r)|u_r|+ g(r^2 u_r)|u_r|^{-1}u_r\cdot \f{d}{dr}u_r \right)\,dx\\
&+\int_{\pa B_1}r|\f{d}{dr}u_r|^2\,\mathcal{H}^{n-1}.
\end{align*}
Here we have used integration by parts in the second step and the formula $\f{d}{dr}u_r=\f{1}{r}(x\cdot \na u_r-2u_r)$ in the last step. Since $u$ satisfies the equation \eqref{ELeq}, direct computation implies that
\beq\label{ELu_r}
\Delta u_r=\left( g(r^2 u_r)\f{u_r}{|u_r|}+D_u g(r^2 u_r)r^{2}|u_r| \right).
\eeq
Substituting \eqref{ELu_r} into the above identity, we obtain
\begin{align*}
&\f{d}{dr}W(u,x_0,r)-\int_{\pa B_1}r|\f{d}{dr}u_r|^2\,\mathcal{H}^{n-1}\\
=& \int_{B_1}\bigg\{ -\left( g(r^2 u_r) \f{u_r}{|u_r|}+D_u g(r^2 u_r)r^{2}|u_r| \right)\f{d}{dr}u_r\\
&\qquad +D_u g(r^2 u_r)\cdot \f{d}{dr}(r^2 u_r)|u_r|+  g(r^2 u_r)\f{u_r}{|u_r|}\cdot \f{d}{dr}u_r\bigg\}\,dx\\
=&2 r \int_{B_1}D_u g(r^2 u_r)\cdot u_r|u_r|\,dx
\end{align*}
Hence we have proved \eqref{Weiss}.
\end{proof}

\begin{rmk}
For other $\al\in[0,2)$, the analogue result still holds for
\beqo
W(u,x_0,r)=\f{1}{r^{n+2\ka-2}}\int_{B_r(x_0)}\f12|\na u|^2+W(u)\,dx-\f{\ka}{2r^{n+2\ka-1}}\int_{\pa B_r(x_0)}|u|^2 \,d\mathcal{H}^{n-1}.
\eeqo
where $\kappa:=\f{2}{2-\al}$ and $W(u)=g(u)|u|^\al$. The derivative of $W(u,x_0,r)$ is given by
\beqo
\f{d}{dr} W(u,x_0,r) =r\int_{\pa B_1}|\f{du_r}{dr}|^2\,d\mathcal{H}^{n-1}+\ka r^{\ka-1}\int_{B_1}D_ug\cdot u_r|u_r|^\al
\eeqo
For our purpose we only need the statement for $\al=1$. The proof for general $\al\in [0,2)$ is identical and we omit it here.
\end{rmk}

\begin{proposition}\label{prop growth}
Let $\al=1$ and $u$ be a bounded entire minimizer and let $\G(u)$ be as defined in \eqref{notations}. There exist constants $r_0$ and $C$, which only depends on $\|u\|_{L^\infty}$ and the potential function $W(u)$, such that
\beq\label{growth estimate}
|u(x)|\leq C \dist(x,\G(u))^2,\quad |\na u(x)|\leq C\dist(x,\G(u))
\eeq
whenever $\dist(x,\G(u))\leq r_0$.
\end{proposition}

\begin{proof}
This proposition and the proof are almost identical to \cite[Theorem 2]{asuw} (except now for a more general potential function). We present the whole argument here for completeness.

The statement of the proposition is equivalent to
\beqo
\sup\limits_{x\in B_r(x_0)}|u(x)|\leq Cr^2,\quad \sup\limits_{x\in B_r(x_0)}|\na u(x)|\leq Cr,
\eeqo
whenever $x_0\in\G(u)$, $r\leq r_0$. By \eqref{ELeq} and the standard theory of elliptic regularity, it suffices to show
\beq\label{growth est int form}
\f{1}{r^n}\int_{B_r(x_0)}|u|\,dx\leq Cr^2,\quad \forall x_0\in \G(u),\; r\leq r_0
\eeq
where $C$ and $r_0$ only depend on $\|u\|_{L^\infty}$ and the potential function $W$.

Note that since $|u|$ is uniformly bounded, we have $u\in C^{1,\g}$, which further implies $|\na u|$ is uniformly bounded. As a result, there is a constant  $r_0$ such that $\dist(x,\G(u))\leq 2r_0$ implies $|u(x)|\leq \theta$, where $\th$ is the constant in Lemma \ref{nondeg}. Also, $W(u(x))$ has the form $g(u(x))|u(x)|$ for some smooth function $g(u)\geq C>0$ when $\dist(x,\G(u))\leq 2r_0$.

Since $|\na u|$ is bounded and $r_0$ is a constant, we have that $W(u,x_0,r_0)$ is uniformly bounded by some constant $C_1$ independent of $u$ and $x_0$. Here $W(u,x_0,r_0)$ is the quantity defined in \eqref{def weiss}. Using Lemma \ref{weiss lemma}, we compute for $r<r_0$
\beq\label{compute integral growth}
\begin{split}
\f{1}{r^{n+2}}\int_{B_r(x_0)}g(u)|u|\,dx& =W(u,x_0,r)-\f{1}{r^{n+2}}\int_{B_r(x_0)}\f12|\na u|^2\,dx\\
&\qquad +\f{1}{r^{n+3}}\int_{\pa B_r(x_0)}|u|^2\,d\ch^{n-1}\\
&=W(u,x_0,r)-\f{1}{r^{n+2}}\int_{B_r(x_0)}\f12|\na (u-p(x-x_0))|^2\,dx\\
&\qquad +\f{1}{r^{n+3}}\int_{\pa B_r(x_0)}|u-p(x-x_0)|^2\,d\ch^{n-1}\\
&\leq W(u,x_0,r_0)+\int_{r}^{r_0}2s\int_{B_1} |D_ug||\f{u(x_0+sx)}{s^2}|^2\,dx\,ds\\
&\qquad +\f{1}{r^{n+3}}\int_{\pa B_r(x_0)}|u-p(x-x_0)|^2\,d\ch^{n-1},
\end{split}
\eeq
for every $p(x)\in\mathcal{H}$, where $\ch$ is defined by
\beqo
\begin{split}
\ch:=&\{p(x): \; p(x)=(p_1(x),...p_m(x)), \text{ each }p_i(x) \text{ is a }\\
 &\quad  \text{homogeneous
harmonic polynomial of second order.} \}
\end{split}
\eeqo
We would like to point out that the homogeneity and harmonicity of $p(x)$ is used in the second equality of \eqref{compute integral growth}.

We already know that the first term in the last step of \eqref{compute integral growth} is bounded by a constant $C_1$ independent of $u$ and $x_0$. For the second term, since $u(x_0+x)\leq C|x|^{\f53}$ when $|x|\leq r_0$ by the $C^{1,\f23}$ regularity (c.f. Remark \ref{c1,alpha remark} and observation below \eqref{notations}), we have
\beqo
\int_{r}^{r_0}2s\int_{B_1} |D_ug||\f{u(x_0+sx)}{s^2}|^2\,dx\,ds\leq C\int_r^{r_0}s^{-3}\int_{B_1} |sx|^{\f{10}{3}}\,dx\,ds\leq C_2
\eeqo
for some constant $C_2$. Because $g(u)\geq C>0$ in $B_r(x_0)$, in order to prove \eqref{growth est int form}, it suffices to show that there is a constant $C_3$, independent of $u$ and $x_0$, such that for any $x_0\in\G(u)$ and $r\leq r_0$,
\beq\label{last term bound}
\min\limits_{p\in\ch} \f{1}{r^{n+3}}\int_{\pa B_r(x_0)}|u-p(x-x_0)|^2\,d\ch^{n-1}\leq C_3.
\eeq

Let $p_{x_0,r}$ be the minimizer of the integral $\int_{\pa B_r(x_0)}|u-p(x-x_0)|^2\,d\ch^{n-1}$ among $p\in\ch$. Then $p_{x_0,r}$ satisfies
\beq\label{ortho}
\int_{\pa_{B_r}(x_0)}(u(x)-p_{x_0,r}(x-x_0))\cdot q(x-x_0)\,d\ch=0\quad \forall q\in\ch.
\eeq

Suppose by contradiction that \eqref{last term bound} is not true, then there is a sequence of entire minimizers $\{u_k\}$ (uniformly bounded), a sequence of points $x_k\in \G(u_k)$ as well as a sequence of radii $r_k\ri 0$ such that
\beqo
M_k:=\f{1}{r_k^{n+3}}\int_{\pa B_{r_k}(x_k)}|u_k-p_{x_k,r_k}(x-x_k)|^2\,d\ch^{n-1}\ri\infty.
\eeqo

Define
\beqo
v_k(x):=\f{u_k(x_k+r_kx)}{r_k^2},\qquad w_k:=\f{v_k-p_{x_k,r_k}}{\sqrt{M_k}}.
\eeqo
Then we immediately get
\beqo
\int_{\pa B_1(0^n)}|w_k|^2\,d\ch^{n-1}=1
\eeqo
and we have
\beq\label{compute w_k}
\begin{split}
&\int_{B_1(0^n)} \f12|\na w_k|^2\,dx-\int_{\pa B_1(0^n)}|w_k|^2\,d\ch^{n-1}\\
=&M_k^{-1}\left( \int_{B_1(0^n)} \f12|\na (v_k-p_{x_k,r_k})|^2\,dx-\int_{\pa B_1(0^n)}|v_k-p_{x_k,r_k}|^2\,d\ch^{n-1} \right)\\
=&M_k^{-1}\left( \int_{B_1(0^n)} \f12|\na v_k|^2\,dx-\int_{\pa B_1(0^n)}|v_k|^2\,d\ch^{n-1} \right)\\
\leq &M_k^{-1} W(u_k,x_k,r_k)\\
\leq & M_k^{-1}\left(  W(u_k,x_k,r_0)+\int_{r_k}^{r_0} 2s \int_{B_1} |D_u g||\f{u_k(x_k+sx)}{s^2}|^2 \,dx\,ds\right)\\
\ri & 0\quad \text{as }k\ri\infty.
\end{split}
\eeq

So $w_k$ is uniformly bounded in $W^{1,2}(B_1)$. Also we note that by \eqref{eleq for al=1} each $w_k$ satisfies the equation
\beqo
\D w_k=\f{1}{\sqrt{M_k}} \left( \f{v_k}{|v_k|}g(u_k)+D_ug(u_k) \right)\chi_{\{|u_k|>0\}},
\eeqo
which implies
\beqo
|\D w_k|\leq \f{C}{\sqrt{M_k}}\ri 0,\quad \text{as }k\ri\infty.
\eeqo
By Schauder estimates, $w_k$ is uniformly bounded in $C_{loc}^{1,\g}(B_1)$ for any $\g< 1$. Therefore we can extract a subsequence, still denoted by $w_k$, that converges to $w_0$ with the following properties
\begin{enumerate}
\itemsep0.5em
\item $w_k\ri w_0$ weakly in $H^1(B_1)$, strongly in $L^2(\pa B_1)$, $\int_{\pa B_1} |w_0|^2\,d\ch^{n-1}=1$.
\item $w_k\ri w_0$ in $C^{1,\g}_{loc}(B_1)$ for any $\g<1$;
\item $\D w_0=0$;
\item $|w_0(0^n)|=|\na w_0(0^n)|=0$;
\item $\int_{\pa B_1} w_0\cdot q\,d\ch^{n-1}=0$ for any $q\in\ch$. This property follows from \eqref{ortho}.
\end{enumerate}

By \cite[Lemma 4.1]{Weiss3}, we know that for any $w_0$ satisfying (3) and (4),
\beqo
\int_{B_1}|\na w_0|^2\,dx\geq 2\int_{\pa B_1}|\na w_0|^2\,d\ch^{n-1}.
\eeqo
On the other hand, from \eqref{compute w_k} we know
\beqo
\int_{B_1}|\na w_0|^2\,dx\leq  2\int_{\pa B_1}|\na w_0|^2\,d\ch^{n-1}.
\eeqo
Therefore we have $\int_{B_1}|\na w_0|^2\,dx= 2\int_{\pa B_1}|\na w_0|^2\,d\ch^{n-1}$ which implies (again by \cite[Lemma 4.1]{Weiss3}) that each component of $w_0$ is a homogeneous harmonic polynomial of second order, i.e. $w_0\in \ch$. This is in contradiction with properties (1) and (5). The proof is complete.
\end{proof}

\section{$(n-1)$-Hausdorff measure of the free boundary for $\al=1$} \label{section 5}

In this section, we continue working with the potential function $W(u)$ satisfying H1 and H2 with $\underline{\al=1}$. Assume $u$ is a bounded entire minimizer of the energy $J$. We would like to study the $(n-1)$-Hausdorff measure of $\pa^* I_0$ and prove the second part of Theorem \ref{main}, i.e. the inequality \eqref{free bdy lower bdd}, \eqref{est free bdy}.

Firstly we focus on the local estimate of $\pa^* \{u=a_i\}$ and we use the same notations and assumptions as in Section \ref{growth est}. Take $a_1=0^m$ and $W(u)=g(u)|u|$ for some $g\in C^2(B_\theta)$, $\th$ as in Lemma \ref{nondeg}. $u$ satisfies the Euler-Lagrange equation \eqref{ELeq}.

Thanks to the growth estimate \eqref{growth est} and the non-degeneracy Lemma \ref{nondeg}, we have for every $x_0\in\G(u)$ (recall that $\G(u)$ is defined in \eqref{notations}) and small $r$,
\begin{align}
\label{control of u}c_1r^2& \leq \sup\limits_{x\in B_r(x_0)}|u(x)|\leq c_2r^2,\\
\label{control of nablau} c_1r &\leq \sup\limits_{x\in B_r(x_0)}|\na u(x)|\leq c_2r.
\end{align}

\begin{thm}\label{local estimate}(Local estimate of $\G$)
There are constants $r_0$ and $C_0$ such that
\beq\label{bdy_meas:local}
\mathcal{H}^{n-1}(\G(u)\cap B_{r_0}(z))\leq C_0\quad  \text{for every }\;z\in \G(u).
\eeq
\end{thm}

\begin{proof}

Take the constant $r_0$ such that for any $x$ that satisfies $\mathrm{dist}(x,\G(u))\leq 2r_0$, $|u(x)|\leq \th$. We will fix a ball $B_{2r_0}(z)$ for some $z\in \G(u)$ in the rest of the proof.

We define
\beqo
v_i:=\pa_{x_i} u,\;\; i=1,2,...,n,\qquad \S_\e(u):=\{x\in B_{r_0}(z)\cap\{|u|>0\}: |\na u|<\e\}.
\eeqo
By differentiating the Euler-Lagrange equation \eqref{ELeq}, formally we have
\begin{equation}\label{equation_for_vi}
\begin{split}
\D v_i=&|u|^{-1}g(u)v_i+ |u|^{-1}(D_ug\cdot v_i)u )-|u|^{-3}g(u)(v_i\cdot u)u\\
&+|u|^{-1}(v_i\cdot u)D_ug+|u|(D^2_u g\cdot v_i).
\end{split}
\end{equation}
Take the function $\psi_\e(x):\BR^+\ri [0,1]$ defined by
\beqo
\psi_\e(x)=\begin{cases}
1, & x\geq \e,\\
\f{x}{\e}, & x\in [0,\e).
\end{cases}
\eeqo
We also choose a smooth cut-off function $\phi \in C_c^{\infty }(B_{2r_0}(z),\BR)$ such that
\beqo
\phi\equiv 1 \text{ in }B_{r_0}(z),\quad |\na \phi|\leq \f{C}{r_0}
\eeqo

Let
\beqo
\ta:=B_{2r_0}(z)\cap \{|u|>0\}.
\eeqo
The key of the proof is to estimate the following integral
\beq\label{finite-perimeter-integral}
I:= \int_{\ta}\na v_i\cdot \na\left[ \psi_\e(|v_i|)\f{v_i}{|v_i|} \phi \right]\,dx,
\eeq
from which estimate \eqref{meas est:Sigma_e} below follows.

\textbf{Claim. }There exists a constant $C(g,r_0)$, which is independent of $\e,\,z$, such that
\beq\label{claim est I}
I\leq C(g,r_0).
\eeq
\begin{proof}[Proof of the Claim]
Define
\beqo
\eta:=\psi_\e(|v_i|)\f{v_i}{|v_i|}\phi
\eeqo
We first show that $\eta\in W_0^{1,2}(B_{2r_0}(z),\BR^m)$. Indeed, by direct computation we have
\begin{align*}
\pa_j \eta&= \psi_\e'(|v_i|)\pa_j|v_i|\f{v_i}{|v_i|}\phi+\psi_\e(|v_i|)\f{\pa_j v_i}{|v_i|}\phi\\
&\quad -\psi_\e(|v_i|)v_i\f{\pa_j|v_i|}{|v_i|^2}\phi+\psi_\e(|v_i|)\f{v_i}{|v_i|}\pa_j\phi
\end{align*}
By definitions of $\psi_\e$, $\phi$ and the $W^{2,2}$ estimate of $u$, the right-hand side is $L^2$-integrable. Combining with the fact that $\phi\in C_0^{\infty}(B_{2r_0}(z))$, we get $\eta\in W_0^{1,2}(B_{2r_0}(z),\BR^m)$.

We notice that since $\D v_i$ is very singular when $|u|\ri 0$, so we can not directly perform integration by parts by moving all the derivatives on $v_i$ in domain $\ta$. Instead, we will switch $\pa_i$ and $\nabla$.

For any $f\in C_0^\infty(B_{2r_0}(z),\BR^m)$, by integration by parts we have
\beqo
\int_{B_{2r_0}(z)}\na v_i\cdot\na f\,dx=\int_{B_{2r_0}(z)} \Delta u\cdot \pa_if\,dx
\eeqo
This can be generalized to the vector-valued function $\eta$ in $W_0^{1,2}(B_{2r_0}(z),\BR^m)$, so we get
\beq\label{ibp}
\int_{\ta}\na v_i\cdot\na \eta\,dx=\int_{B_{2r_0}(z)} \na v_i\cdot\na \eta\,dx=\int_{B_{2r_0}(z)} \D u\cdot \pa_i\eta\,dx=\int_{\ta} \D u\cdot \pa_i\eta\,dx
\eeq
Above we have exploited the fact that $D^2 u$ and $\na \eta$ vanish almost everywhere on $\{|u|=0\}$. So it suffices to prove
\beq\label{claim after ibp}
\int_{\ta} \D u\cdot \pa_i\eta \,dx\leq C(g,r_0).
\eeq

We define the set $\ta_\delta:=B_{2r_0}(z)\cap \{|u|>\d\}$, it is obvious that $\ta_\d\subset \ta$ for any $\d>0$ and $\ta=\lim_{\d\ri 0 } \ta_\d$. Then we have
\begin{align}
\nonumber &\int_{\ta} \D u\cdot \pa_i\eta\,dx\\
\nonumber=&\lim\limits_{\d\ri 0} \int_{\ta_\d} \D u\cdot\pa_i \eta\,dx\\
\label{two parts}=&\lim\limits_{\d\ri 0} \int_{\ta_\d} -\D v_i\cdot \eta\,dx+\lim\limits_{\d\ri 0} \int_{\pa \ta_\d} \Delta u\cdot \eta \g_i d\s
\end{align}

For the first term in \eqref{two parts}, we further compute
\begin{align}
\nonumber &-\int_{\ta_\d} \D v_i\cdot [\psi_\e(|v_i|)\f{v_i}{|v_i|}\phi]\,dx\\
\label{estimate:integral}=&-\int_{\ta_\d} \psi_{\e}(|v_i|)|v_i|^{-1}\phi \bigg( |u|^{-1}g(u)|v_i|^2-|u|^{-3}g(u)(v_i\cdot u)^2\\
\nonumber &\qquad\qquad\qquad \qquad  +2|u|^{-1}(D_ug\cdot v_i)(u\cdot v_i)+|u|(v_i\cdot D^2_ug\cdot v_i) \bigg)\,dx
\end{align}

By the Cauchy-Schwartz inequality,
\beqo
|u|^{-1}g(u)|v_i|^2-|u|^{-3}g(u)(v_i\cdot u)^2\geq 0.
\eeqo
Substituting this into \eqref{estimate:integral} gives
\beq
\begin{split}
\label{First part:bounded by C}&\int_{\ta_\d}-\D v_i\cdot \eta\,dx \\
\leq &\bigg| \int_{\ta_\d} \psi_{\e}(|v_i|)|v_i|^{-1}\phi \bigg( 2|u|^{-1}(D_ug\cdot v_i)(u\cdot v_i)+|u|(v_i\cdot D^2_ug\cdot v_i) \bigg)\,dx\bigg|\\ \leq &C(g, r_0).
\end{split}
\eeq
The integral is bounded by a constant $C(g,r_0)$ (doesn't depend on the choice of $z, \,\d,\, \e$) because $|v_i|$, $D_u(g)$, $D_u^2g$, $u$ are all uniformly bounded by a constant in $B_{2r_0}(z)$.

For the second part in \eqref{two parts}, we apply \eqref{ELeq} to obtain
\beq
\begin{split}\label{second part}
&\int_{\pa \ta_\d} \D u\cdot \eta\gamma_i d\s\\
=&\int_{\pa\{|u|>\d\}\cap B_{2r_0}(z)}\D u\cdot \eta\g_i d\s\\
=&\int_{\pa\{|u|>\d\}\cap B_{2r_0}(z)}\left(g(u)\f{u}{|u|}+|u|D_ug(u)\right)\left( \psi_\e(|v_i|)\f{v_i}{|v_i|}\phi) \right) \g_i d\s\\
=&\int_{\pa\{|u|>\d\}\cap B_{2r_0}(z)} g(u)\pa_i|u|\f{\psi_\e(|v_i|)}{|v_i|} \phi \g_i d\s+ \int_{\pa\{|u|>\d\}\cap B_{2r_0}(z)}|u|\pa_i g(u)\f{\psi_\e(|v_i|)}{|v_i|} \phi \g_i d\s\\
=&:\mathrm{I}+\mathrm{II}
\end{split}
\eeq

We notice that on $\pa\{|u|>\delta\}$, if $\left|\nabla |u|\right|\neq 0$, then the outer normal vector can be written as $\g=\f{-\na |u|}{\left| \na |u|\right|}$, so we obtain that $\mathrm{I}\leq 0$.

For the term $\mathrm{II}$, we perform integration by parts again to get
\beq
\begin{split}\label{estimate:bdy int}
\lim\limits_{\d\ri 0}\mathrm{II}\leq &\lim\limits_{\d\ri 0}\d
\left|\int_{\pa\{|u|>  \delta\}\cap B_{2r_0}(z)} \pa_ig(u)\f{\psi_\e(|v_i|)}{|v_i|}\phi\g_i\,d\sigma\right|\\
\leq &\lim\limits_{\d\ri 0}\delta \int_{\{|u|>\delta|\cap B_{2r_0}(z)\}}\left|\pa_i(\pa_ig(u)\f{\psi_\e(|v_i|)}{|v_i|}\phi)\right|\,dx=0
\end{split}
\eeq
We note that in the last step of \eqref{estimate:bdy int}, the limit is zero since it is the multiplication of $\d$ (goes to zero) and a bounded integral (the bound depends on $\e$, but doesn't depend on $\d$).

Combining \eqref{estimate:integral}, \eqref{First part:bounded by C}, \eqref{second part} and \eqref{estimate:bdy int} will conclude the proof of the Claim.

\end{proof}

On the other hand, we compute
\begin{align*}
I&=\int_{\ta}\big(\na v_i \na\psi_\e(|v_i|)\f{v_i}{|v_i|}\phi\big)+\big( \na v_i\cdot \na(\f{v_i}{|v_i|})\psi_\e(|v_i|)\phi \big)\\
&\qquad\qquad\qquad  +\big( \na v_i \f{v_i}{|v_i|}\na\phi \psi_\e(|v_i|) \big)\,dx\\
&=\f{C}{\e}\int_{\ta\cap\{0<|v_i|<\e\}} |\na |v_i||^2\phi\,dx\\
&\qquad + \int_{\ta\cap \{|v_i|>0\}} \left( |v_i|^{-1}|\na v_i|^2- |v_i|^{-1}|\na|v_i||^2\right)\psi_\e(|v_i|)\phi \,dx\\
&\qquad +\int_{\ta}\left( \na|v_i|\na\phi\psi_\e(|v_i|) \right)\,dx
\end{align*}
Note that we have
\beqo
\begin{split}
&\qquad\qquad\qquad |v_i|^{-1}|\na v_i|^2- |v_i|^{-1}|\na|v_i||^2\geq 0,\\
&\int_{\ta}\left( \na|v_i|\na\phi\psi_\e(|v_i|) \right)\,dx\leq (\int_{\ta}|\na |v_i||^2)^{\f12}(\int_{\ta} |\na \phi\psi_\e(|v_i|)|^2)^{\f12}\leq C(r_0).
\end{split}
\eeqo
Combining with \eqref{claim est I}, we conclude that
\beq\label{meas est:Sigma_e}
\int_{\S_\e}|\na|v_i||^2\,dx\leq C(g,r_0)\e, \quad \forall \e<<1.
\eeq

We need the following lemma.
\begin{lemma}\label{vi energy}
There are constants $\e_0$ and $C$ such that for every $\e\leq \e_0$,
\beqo
\sum\limits_{i=1}^n\int_{B_\e(z)\cap\Om(u)}|\na|v_i||^2\,dx\geq C\cl^n(B_{\e}(z)),\quad \forall z\in\G(u).
\eeqo
\end{lemma}
\begin{proof}
Recall that $\Om(u)=\{|u|>0\}$. If the statement is false, there exists a sequence of uniformly bounded entire minimizers $\{u_j\}_{j=1}^\infty$, $\{z_j\in\G(u_j)\}$,  $\{\e_j\}$ as well as $\{C_j\}$ such that
\begin{align}
&\nonumber\qquad\qquad \lim\limits_{j\ri\infty}\e_j=0,\quad \lim\limits_{j\ri\infty} C_j=0,\\
&\label{contradict1}\sum\limits_{i=1}^n\int_{B_{\e_j}(z_j)\cap\Om(u)}|\na|\pa_iu_j||^2\,dx< C_j\cl^n(B_{\e_j}(z_j)).
\end{align}
Define $f_j(x): B_1(0^n)\ri \BR^m$ as
\beqo
f_j(x):=\f{u_j(z_j+\e_jx)}{\e_j^2}.
\eeqo
Then by Proposition \ref{prop growth}, \eqref{control of u} and \eqref{contradict1}, we have
\begin{enumerate}
\itemsep0.5em
\item $|f_j(0^n)|=|\na f_j(0^n)|=0\quad \text{for every }j$,
\item $\|f_j\|_{C^{1,\g}(B_1(0^n))} (\g<1)$ is uniformly bounded.
\item $\sum\limits_i^n \int_{B_1}|\na |\pa_i f_j||^2\,dx\leq C_j\om_n,$
\item $\sup_{B_{\f12}(0^n)} |f_j(x)|\geq C>0$ for some constant $C$.
\end{enumerate}
Using all these properties, we can get the following convergence results up to some subsequence,
\begin{align}
&\nonumber \qquad\qquad\quad f_j\ri f \text{ in }C^{1}(B_{1}(0^n)),\\
&\label{convergence1}|f(0^n)|=|\na f(0^n)|=0,\quad \sum\limits_i^n \int_{B_1}|\na |\pa_i f||^2\,dx=0. \\
& \label{convergence2}\quad \qquad\qquad \sup_{B_{\f12}(0^n)} |f(x)|\geq C>0.
\end{align}
Note that \eqref{convergence1} implies that $f\equiv 0$ in $B_1(0^n)$, which yields a contradiction with \eqref{convergence2}. The proof of Lemma \ref{vi energy} is complete.
\end{proof}

According to Besicovitch covering lemma, we can find a covering of $\G(u)\cap B_{r_0}(z)$ by a finite family of balls $\{B_j\}_{j\in J}$, such that each ball is of radius $\e$ and centered on $\G(u)\cap B_{r_0}$, and no more than $N_n$ balls from this family overlap, where $N_n$ is independent of $\e$ and of the set $\G(u)\cap B_{r_0}$. By the estimate \eqref{control of nablau}, we have $B_{\e}(z)\cap\Om(u)\subset \S_{C\e}$ for some constant $C$. Consequently, we obtain
\begin{align*}
\sum\limits_{j\in J}\cl^n(B_j)&\leq C \sum\limits_{j\in J}\sum\limits_{i=1}^n\int_{B_j\cap\{|u|>0\}}|\na(|v_i|)|^2\,dx\quad \text{(by Lemma \ref{vi energy})}\\
&\leq C\sum\limits_{i=1}^n\int_{\S_{C\e}}|\na(|v_i|)|^2\,dx\leq C(g,r_0)\e \quad \text{(by \eqref{meas est:Sigma_e}).}
\end{align*}
This implies
\beqo
\sum\limits_{j\in J}\e^{n-1}\leq C(g,r_0),
\eeqo
for some constant $C(n,r_0)$ independent of the choice of $z$. Finally letting $\e\ri 0$, we get
\beqo
\mathcal{H}^{n-1}(\G(u)\cap B_{r_0}(z))\leq C(g,r_0).
\eeqo
The proof of Theorem \ref{local estimate} is complete.
\end{proof}

\begin{rmk}\label{local_rmk}
Using Theorem \ref{local estimate}, we can prove that for any $R$, there exists $C(R)$ such that
\beqo
\ch^{n-1}(B_R(x)\cap \G(u))\leq C(R)
\eeqo
To prove this, one simply covers the set $B_R(x)\cap \G(u)$ by identical small balls $\{B_{r_0}(z_i)\}$ such that $z_i\in \G(u)$ for every $i$. We omit the details.
\end{rmk}

Now we use the local estimate Remark \ref{local_rmk} and Theorem \ref{first part main theorem} to prove the global estimate of the $\mathcal{H}^{n-1}$ measure of the free boundary $\pa^* I_0$.

\begin{thm}[Second part of Theorem \ref{main}]\label{global estimate}
Let $\al\in (0,2)$, $x_0\in \BR^n$. There are constants $c_1,r_0$ such that for any $r>r_0$,
\beqo
\mathcal{H}^{n-1}(\pa^* I_0 \cap B_r(x_0))\geq c_1r^{n-1}.
\eeqo
And when $\al=1$, there are constants $c_2, r_0$ such that for $r\geq r_0$,
\beqo
\mathcal{H}^{n-1}(\pa^* I_0 \cap B_r(x_0))\leq c_2r^{n-1}.
\eeqo
\end{thm}

\begin{rmk}
Unlike the local estimate, here all the constants depend on $u$.
\end{rmk}

\begin{proof}
Without loss of generality we take $x_0=0^n$. According to Theorem \ref{two phase existence}, for sufficiently large $r$, there are two phases $a_1,a_2\in A$, which depend on $r$,  such that
\beqo
\cl^n(B_r\cap \{u=a_j\})\geq cr^n, \quad j=1,2.
\eeqo
Using the relative isoperimetric inequality we obtain that
\beqo
\begin{split}
&\ch^{n-1}(\pa^*\{|u-a_1|>0\}\cap B_r)\\
\geq &C\left(\min\{\cl^n(B_r\cap \{u=a_1\}), \cl^n(B_r\backslash \{u=a_1\} )\} \right)^{\f{n-1}{n}}\geq c_1r^{n-1},
\end{split}
\eeqo
which gives the lower bound. Note that this estimate is valid for any $0<\al<2$.

For the upper bound, we fix $\al=1$ and examine more closely the proof of Theorem \ref{first part main theorem}. Again we consider the domain $\tilde{S}_{kL}$ and classify all the sub-cubes $\{S_i\}_1^{(2k)^n}$ into five classes $T_1$--$T_5$. If $S_i\in T_4$, then $u(x)\equiv a_{j_0}$ for all $x\in \overline{S_i}$, which implies $\ch^{n-1}(S_i\cap \pa^* I_0)=0$. Moreover, for any $x_0\in\pa S_i$, by the definition of $T_4$ it holds that
\beqo
\max\limits_{x\in B_L(x_0)} |u(x)-a_{j_0}|\leq \theta.
\eeqo

By the proof of Lemma \ref{value of L}, we obtain that $B_{L/4}(x_0)\subset \{u=a_{j_0}\}$ and consequently $x_0\not\in \pa^* I_0$. As a result, we have
\beqo
\ch^{n-1}(\overline{S_i}\cap \pa^* I_0)=0.
\eeqo

Using estimates in Theorem \ref{first part main theorem} and Remark \ref{local_rmk}, we have for large enough $k$
\beq\label{upper bd free bdy}
\begin{split}
&\ch^{n-1}(\pa^*I_0\cap \tilde{S}_{kL})\\
\leq & \sum\limits_{S_i\in T_1\cup T_2\cup T_3\cup T_5} \ch^{n-1}(\pa^*I_0\cap \overline{S_i}) \\
\leq &(|T_1|+|T_2|+|T_3|+|T_5|) C(L)\\
\leq & c_2(W,u) k^{n-1}.
\end{split}
\eeq
The upper bound follows immediately from \eqref{upper bd free bdy} and the proof is complete.

\end{proof}

\end{document}